\documentclass[11pt]{amsart}
\usepackage{amsmath,amssymb}
\usepackage{verbatim}
\usepackage{color}
\usepackage{hyperref}

\begin{document}

\newtheorem{thm}{Theorem}[section]
\newtheorem{theorem}{Theorem}[section]
\newtheorem{lem}[thm]{Lemma}
\newtheorem{lemma}[thm]{Lemma}
\newtheorem{prop}[thm]{Proposition}
\newtheorem{proposition}[thm]{Proposition}
\newtheorem{cor}[thm]{Corollary}
\newtheorem{defn}[thm]{Definition}
\newtheorem{remark}[thm]{Remark}
\newtheorem{conj}[thm]{Conjecture}

\numberwithin{equation}{section}

\newcommand{\Z}{{\mathbb Z}} 
\newcommand{\Q}{{\mathbb Q}}
\newcommand{\R}{{\mathbb R}}
\newcommand{\C}{{\mathbb C}}
\newcommand{\N}{{\mathbb N}}
\newcommand{\FF}{{\mathbb F}}
\newcommand{\fq}{\mathbb{F}_q}
\newcommand{\rmk}[1]{\footnote{{\bf Comment:} #1}}

\newcommand{\bfA}{{\boldsymbol{A}}}
\newcommand{\bfY}{{\boldsymbol{Y}}}
\newcommand{\bfX}{{\boldsymbol{X}}}
\newcommand{\bfZ}{{\boldsymbol{Z}}}
\newcommand{\bfa}{{\boldsymbol{a}}}
\newcommand{\bfy}{{\boldsymbol{y}}}
\newcommand{\bfx}{{\boldsymbol{x}}}
\newcommand{\bfz}{{\boldsymbol{z}}}
\newcommand{\F}{\mathcal{F}}
\newcommand{\Gal}{\mathrm{Gal}}
\newcommand{\Fr}{\mathrm{Fr}}
\newcommand{\Hom}{\mathrm{Hom}}
\newcommand{\GL}{\mathrm{GL}}

\renewcommand{\mod}{\;\operatorname{mod}}
\newcommand{\ord}{\operatorname{ord}}
\newcommand{\TT}{\mathbb{T}}
\renewcommand{\i}{{\mathrm{i}}}
\renewcommand{\d}{{\mathrm{d}}}
\renewcommand{\^}{\widehat}
\newcommand{\HH}{\mathbb H}
\newcommand{\Vol}{\operatorname{vol}}
\newcommand{\area}{\operatorname{area}}
\newcommand{\tr}{\operatorname{tr}}
\newcommand{\norm}{\mathcal N} 
\newcommand{\intinf}{\int_{-\infty}^\infty}
\newcommand{\ave}[1]{\left\langle#1\right\rangle} 
\newcommand{\Var}{\operatorname{Var}}
\newcommand{\Prob}{\operatorname{Prob}}
\newcommand{\sym}{\operatorname{Sym}}
\newcommand{\disc}{\operatorname{disc}}
\newcommand{\CA}{{\mathcal C}_A}
\newcommand{\cond}{\operatorname{cond}} 
\newcommand{\lcm}{\operatorname{lcm}}
\newcommand{\Kl}{\operatorname{Kl}} 
\newcommand{\leg}[2]{\left( \frac{#1}{#2} \right)}  
\newcommand{\Li}{\operatorname{Li}}

\newcommand{\sumstar}{\sideset \and^{*} \to \sum}

\newcommand{\LL}{\mathcal L} 
\newcommand{\sumf}{\sum^\flat}
\newcommand{\Hgev}{\mathcal H_{2g+2,q}}
\newcommand{\USp}{\operatorname{USp}}
\newcommand{\conv}{*}
\newcommand{\dist} {\operatorname{dist}}
\newcommand{\CF}{c_0} 
\newcommand{\kerp}{\mathcal K}

\newcommand{\Cov}{\operatorname{cov}}
\newcommand{\Sym}{\operatorname{Sym}}

\newcommand{\ES}{\mathcal S} 
\newcommand{\EN}{\mathcal N} 
\newcommand{\EM}{\mathcal M} 
\newcommand{\Sc}{\operatorname{Sc}} 
\newcommand{\Ht}{\operatorname{Ht}}

\newcommand{\E}{\operatorname{E}} 
\newcommand{\sign}{\operatorname{sign}} 

\newcommand{\divid}{d} 

\title[Rudnick and Soundararajan's Theorem for Function Fields]
{Rudnick and Soundararajan's Theorem for Function Fields}
\author{Julio Andrade}
\address{Mathematical Institute, University of Oxford, Radcliffe Observatory Quarter, Woodstock Road, Oxford, OX2 6GG, UK}
\email{buenodeandra@maths.ox.ac.uk}


\subjclass[2010]{Primary 11M38; Secondary 11G20, 11M50, 14G10}
\keywords{function fields \and finite fields \and hyperelliptic curves \and lower bounds for moments \and moments of $L$--functions \and quadratic Dirichlet $L$--functions \and random matrix theory}

\begin{abstract}
In this paper we prove a function field version 
of a theorem by Rudnick and Soundararajan about 
lower bounds for moments of quadratic Dirichlet 
$L$--functions. We establish lower bounds for the 
moments of quadratic Dirichlet $L$--functions 
associated to hyperelliptic curves of genus $g$ 
over a fixed finite field $\mathbb{F}_{q}$ in the 
large genus $g$ limit.
\end{abstract}
\date{\today}

\maketitle

\section{Introduction}

It is a fundamental problem in 
analytic number theory to estimate
moments of central values of $L$--functions 
in families. For example, in the case of 
the Riemann zeta function the question 
is to establish asymptotic formulae for

\begin{equation}
\label{1.1}
M_{k}(T):=\int_{1}^{T}|\zeta(\tfrac{1}{2}+it)|^{2k}dt,
\end{equation} 
where $k$ is a positive integer and $T\rightarrow\infty$.

A believed folklore conjecture asserts that, 
as $T\rightarrow\infty$, there is a positive 
constant $C_{k}$ such that

\begin{equation}
\label{1.2}
M_{k}(T)\sim C_{k}T(\log T)^{k^{2}}.
\end{equation}   
Due to the work of Conrey and Ghosh \cite{Con-Gho} 
the conjecture above assumes a more explicit form, namely

\begin{equation}
\label{1.3}
C_{k}=\frac{a_{k}g_{k}}{\Gamma(k^{2}+1)},
\end{equation}
where 

\begin{equation}
\label{1.4}
a_{k}=\prod_{p \ \text{prime}}\left[\left(1-\frac{1}{p}\right)^{k^{2}}\sum_{m\geq0}\frac{d_{k}(m)^{2}}{p^{m}}\right],
\end{equation}
$g_{k}$ is an integer when $k$ is an integer 
and $d_{k}(n)$ is the number of ways to 
represent $n$ as a product of $k$ factors.

Asymptotics for $M_{k}(T)$ are only known 
for $k=1$, due to Hardy and Littlewood \cite{Har-Lit}

\begin{equation}
\label{1.5}
M_{1}(T)\sim T\log T,
\end{equation}
and for $k=2$, due to Ingham \cite{Ing}

\begin{equation}
\label{1.6}
M_{2}(T)\sim\frac{1}{2\pi^{2}}T\log^{4}T.
\end{equation}
Unfortunately the recent technology does not allow us to obtain asymptotics for higher moments of the Riemann zeta function. The same statement applies for the higher moments of other $L$--functions. However, due to the precursor work of Keating and Snaith \cite{Kea-Sna1,Kea-Sna2} and, subsequently, due to the work of Conrey, Farmer, Keating, Rubinstein and Snaith \cite{CFKRS}, and Diaconu, Goldfeld and Hoffstein \cite{DGH}, there are now very elegant conjectures for moments of $L$--functions.

The work of Katz and Sarnak \cite{Ka-Sa1,Ka-Sa2} associates a symmetry group for each family of $L$--function and the moments are sensitive and take different forms for each one of these groups. In other words the conjectured asymptotic formulas for the moments of families of $L$--function depends whether the symmetry group attached to the family is unitary, orthogonal or symplectic. For a recent and detailed discussion about a working definition of a family of $L$--functions see \cite{SST}.

We will typify the conjectures above by considering different families of $L$--functions. For example, the family of all Dirichlet $L$--functions $L(s,\chi)$, as $\chi$ varies over primitive characters $(\bmod \ q)$, is an example of a unitary family, and it is conjectured that

\begin{equation}
\label{1.7}
\sideset{}{^{*}}\sum_{\chi(\bmod q)}|L(\tfrac{1}{2},\chi)|^{2k}\sim C_{U(N)}(k)q(\log q)^{k^{2}},
\end{equation}  
where $k\in\mathbb{N}$ and $C_{U(N)}(k)$ is a positive constant. For a symplectic family of $L$--functions we consider the quadratic Dirichlet $L$--functions $L(s,\chi_{d})$ associated to the quadratic character $\chi_{d}$, as $d$ varies over fundamental discriminants. In this case it is conjectured that

\begin{equation}
\label{1.8}
\sideset{}{^{\flat}}\sum_{|d|\leq X}L(\tfrac{1}{2},\chi_{d})^{k}\sim C_{USp(2N)}(k)X(\log X)^{k(k+1)/2},
\end{equation}
where $k\in\mathbb{N}$ and $C_{USp(2N)}(k)$ is a positive constant. And finally we consider the family of $L$--functions associated to Hecke eigencuspforms $f$ of weight $k$ for the full modular group $SL(2,\mathbb{Z})$ as $f$ varies in the set $H_{k}$ of Hecke eigencuspforms. This is an example of an orthogonal family and it is conjectured that

\begin{equation}
\label{1.9}
\sideset{}{^{h}}\sum_{f\in H_{k}}L(\tfrac{1}{2},f)^{r}\sim C_{O(N)}(r)(\log k)^{r(r-1)/2},
\end{equation}
where $C_{O(N)}(r)$ is a positive constant, $k\equiv0(\bmod \ 4)$ and

\begin{equation}
\label{1.10}
\sideset{}{^{h}}\sum_{f\in H_{k}}L(\tfrac{1}{2},f)^{r}:=\sum_{f\in H_{k}}\frac{1}{\omega_{f}}L(\tfrac{1}{2},f)^{r},
\end{equation}
with

\begin{equation}
\label{1.11}
\omega_{f}:=\frac{(4\pi)^{k-1}}{\Gamma(k-1)}\left\langle f,f\right\rangle=\frac{k-1}{2\pi^{2}}L(1,\text{Sym}^{2}f),
\end{equation}
where $\left\langle f,f\right\rangle$ denotes the Petersson inner product. For more details on Hecke eigencuspforms $L$--functions see Iwaniec \cite{Iwa}.

The conjectures \eqref{1.2}, \eqref{1.7} and \eqref{1.8} can be verified for small values of $k$ and the same holds for \eqref{1.9}, where it can be verified only for small values of $r$. Ramachandra \cite{Ram} showed that

\begin{equation}
\label{1.12}
\int_{1}^{T}|\zeta(\tfrac{1}{2}+it)|^{2k}dt\gg T(\log T)^{k^{2}},
\end{equation}
for positive integers $k$. Titchmarsh \cite[Theorem 7.19]{Tit} had proved a smooth version of these lower bound for positive integer $k$. The work of Heath--Brown \cite{HB} extends \eqref{1.12} for all positive rational numbers $k$. Recently Radziwi\l\l \ and Soundararajan \cite{Ra-So} proved that

\begin{equation}
\label{1.13}
M_{k}(T)\geq e^{-30k^{4}}T(\log T)^{k^{2}},
\end{equation}
for any real number $k>1$ and all large $T$. For other families of $L$--functions, as those given above, the lower bounds for moments were proved by Rudnick and Soundarajan in \cite{Ru-So1,Ru-So2} where they have established that

\begin{equation}
\label{1.14}
\sideset{}{^{*}}\sum_{\substack{\chi(\bmod q) \\ \chi\neq\chi_{0}}}|L(\tfrac{1}{2},\chi)|^{2k}\gg_{k}q(\log q)^{k^{2}},
\end{equation}
for a fixed natural number $k$ and all large primes $q$. They also proved in \cite{Ru-So2} that

\begin{equation}
\label{1.15}
\sideset{}{^{h}}\sum_{f\in H_{k}}L(\tfrac{1}{2},f)^{r}\gg_{r}(\log k)^{r(r-1)/2},
\end{equation}
for any given natural number $r$, and weight $k\geq12$ with $k\equiv0(\bmod \ 4)$. And for the symplectic family they showed that for every even natural number $k$ 

\begin{equation}
\label{1.16}
\sideset{}{^{\flat}}\sum_{|d|\leq X}L(\tfrac{1}{2},\chi_{d})^{k}\gg_{k}X(\log X)^{k(k+1)/2},
\end{equation}
where the sum is taken over fundamental discriminants $d$. Radziwi\l\l \ and Soundararajan \cite{Ra-So} pointed out that their method may easily be modified to provide lower bounds for moments to the case of $L$--functions in families, for any real number $k>1$.

Recently, in a beautiful paper, Tamam \cite{Ta} proved the function field analogue of \eqref{1.14}. In this paper we consider the function field analogue of equation \eqref{1.16} for quadratic Dirichlet $L$--functions associated to a family of hyperelliptic curves over $\mathbb{F}_{q}$. See next section.

\section{Main Theorem}

Before we enunciate the main theorem of this paper we need a few basic 
facts about rational function fields. We start by fixing
a finite field $\mathbb{F}_{q}$ of odd cardinality $q=p^{a}$
with $p$ a prime. And we denote by $A=\mathbb{F}_{q}[T]$ the
polynomial ring over $\mathbb{F}_{q}$ and by $k=\mathbb{F}_{q}(T)$ the
rational function field over $\mathbb{F}_{q}$.

The zeta function associated to $A$ is defined by the following 
Dirichlet series

\begin{equation}
\label{2.1}
\zeta_{A}(s):=\sum_{\substack{f\in A \\ \text{monic}}}\frac{1}{|f|^{s}} \ \ \ \ \ \ \ \ \text{for Re$(s)>1$},
\end{equation}
where $|f|=q^{\text{deg}(f)}$ for $f\neq0$ and $|f|=0$ for $f=0$. Surprisingly
the zeta function associated to $A$ is a much simpler object than the usual Riemann zeta function and can be showed that

\begin{equation}
\label{2.2}
\zeta_{A}(s)=\frac{1}{1-q^{1-s}}.
\end{equation}

Let $D$ be a square--free monic polynomial in $A$ of degree odd. Then we define the quadratic character $\chi_{D}$ attached to $D$ by making use of the quadratic residue symbol for $\mathbb{F}_{q}[T]$ by

\begin{equation}
\label{2.3}
\chi_{D}(f)=\left(\frac{D}{f}\right).
\end{equation}
In other words, if $P\in A$ is monic irreducible we have

\begin{equation}
\label{2.4}
\chi_{D}(P)=\Bigg\{
\begin{array}{cl}
0, & \mathrm{if}\ P\mid D,\\
1, & \mathrm{if}\ P\not{|} D \ \mathrm{and} \ D \ \mathrm{is \ a \ square \ modulo} \ P,\\
-1, & \mathrm{if}\ P\not{|} D \ \mathrm{and} \ D \ \mathrm{is \ a \ non\ square \ modulo} \ P.\\
\end{array}
\end{equation}
For more details about Dirichlet characters for function fields see \cite[Chapter 3]{Ros} and \cite{Fai-Rud}.

We attach to the character $\chi_{D}$ the quadratic Dirichlet $L$--function defined by

\begin{equation}
\label{2.5}
L(s,\chi_{D})=\sum_{\substack{f\in A \\ f \ \text{monic}}}\frac{\chi_{D}(f)}{|f|^{s}} \ \ \ \ \ \ \ \ \text{for Re$(s)>1$}.
\end{equation}

If $D\in\mathcal{H}_{2g+1,q}$, where 

\begin{equation}
\label{2.6}
\mathcal{H}_{2g+1,q}=\{\text{$D\in A$, square--free, monic and deg$(D)=2g+1$}\},
\end{equation}
then the $L$--function associated to $\chi_{D}$ is indeed the numerator of the zeta function associated to the hyperelliptic curve $C_{D}:y^{2}=D(T)$ and therefore $L(s,\chi_{D})$ is a polynomial in $u=q^{-s}$ of degree $2g$ given by

\begin{equation}
\label{2.7}
L(s,\chi_{D})=\sum_{n=0}^{2g}\sum_{\substack{f \ \text{monic} \\ \text{deg}(f)=n}}\chi_{D}(f)q^{-ns}.
\end{equation}
(see \cite[Propositions 14.6 and 17.7]{Ros} and \cite[Section 3]{And-Kea}).

This $L$--function satisfies a functional equation, namely

\begin{equation}
\label{2.8}
L(s,\chi_{D})=(q^{1-2s})^{g}L(1-s,\chi_{D}),
\end{equation}
and the Riemann hypothesis for curves proved by Weil \cite{Wei} tell us that all the zeros of $L(s,\chi_{D})$ have real part equals $1/2$. 

The main result of this paper is now presented:

\begin{theorem}
\label{thm:main}
For every even natural number $k$ we have,

\begin{equation}
\frac{1}{\#\mathcal{H}_{2g+1,q}}\sum_{D\in\mathcal{H}_{2g+1,q}}L(\tfrac{1}{2},\chi_{D})^{k}\gg_{k}(\log_{q}|D|)^{k(k+1)/2}.
\end{equation}
\end{theorem}

\begin{remark}
To avoid any misunderstanding concerning the notation and conventions presented in this paper it is necessary a note about the notation used in the theorem above and in the rest of this note. On the formula above the right-hand side of the main lower bound appears $|D|=q^{2g+1}$ while $D$ is the summation variable on the left-hand side of that same formula. This is done because the function $D\mapsto|D|$ is constant within $\mathcal{H}_{2g+1,q}$ and so we can always write 

$$\sum_{D\in\mathcal{H}_{2g+1,q}}|D|=|D|\sum_{D\in\mathcal{H}_{2g+1,q}}1.$$
In the case the reader feel uncomfortable with the above notation he/she can always remember that $|D|=q^{2g+1}$.
\end{remark}

\begin{remark}
\label{rmk1}
For simplicity, we will restrict ourselves to the fundamental discriminants $D\in A$, $D$ monic and $\text{deg}(D)=2g+1$. But the calculations are analogous for the even case, i.e., $\text{deg}(D)=2g+2$. 
\end{remark}

Using the same techniques developed by Rudnick and Soundararajan in \cite{Ru-So1, Ru-So2} and extended for function fields in this paper we can also prove the following theorem.

\begin{theorem}
For every even natural number $k$ and $n=2g+1$ or $n=2g+2$ we have,

\begin{equation}
\frac{1}{\pi_{A}(n)}\sum_{\substack{P \ \text{monic} \\ \text{irreducible} \\ \text{deg}(P)=n}}L(\tfrac{1}{2},
\chi_{P})^{k}\gg_{k}(\log_{q}|P|)^{\tfrac{k(k+1)}{2}},
\end{equation}
where $\pi_{A}(n)=\#\left\{\text{$P\in\mathbb{F}_{q}[T]$ monic and irreducible, deg$(P)=n$}\right\}$ and the prime number theorem for polynomials \cite[Theorem 2.2]{Ros} says that $\pi_{A}(n)=\frac{q^{n}}{n}+O\left(\frac{q^{n/2}}{n}\right)$.
\end{theorem}

\section{Necessary Tools}

In this section we present some auxiliary lemmas that will be used in the proof of the main theorem. We start with:

\begin{lemma}[``Approximate" Functional Equation]
\label{lem1}
Let $D\in\mathcal{H}_{2g+1,q}$. Then $L(s,\chi_{D})$ can be represented as

\begin{equation}
\label{3.1}
L(s,\chi_{D})=\sum_{\substack{f_{1} \ \text{monic} \\ \text{deg}(f_{1})\leq g}}\frac{\chi_{D}(f_{1})}{|f_{1}|^{s}}+(q^{1-2s})^{g}\sum_{\substack{f_{2} \ \text{monic} \\ \text{deg}(f_{2})\leq g-1}}\frac{\chi_{D}(f_{2})}{|f_{2}|^{1-s}}.
\end{equation}
\end{lemma}

\begin{proof}
The proof of this Lemma can be found in \cite[Lemma 3.3]{And-Kea}.
\end{proof}

The following lemma is the function field analogue of P\'{o}lya--Vinogradov inequality for character sums.

\begin{lemma}[P\'{o}lya--Vinogradov inequality for $\mathbb{F}_{q}(T)$]
\label{lem2}
Let $\chi$ be a non--principal Dirichlet character modulo $Q\in\mathbb{F}_{q}[T]$ such that $\text{deg}(Q)$ is odd. Then we have,

\begin{equation}
\label{3.2}
\sum_{\text{deg}(f)=x}\chi(f)\ll|Q|^{1/2}.
\end{equation}
\end{lemma}

\begin{proof}
The proof of this Lemma can be found in \cite[Proposition 2.1]{Hsu}.
\end{proof}

The next lemma is taken from Andrade-Keating \cite[Proposition 5.2]{And-Kea} and it is about counting the number of square--free polynomials coprime to a fixed monic polynomial.

\begin{lemma}
\label{lem3}
Let $f\in A$ be a fixed monic polynomial. Then for all $\varepsilon>0$ we have that

\begin{equation}
\label{3.3}
\sum_{\substack{D\in\mathcal{H}_{2g+1,q} \\ (D,f)=1}}1=\frac{|D|}{\zeta_{A}(2)}\prod_{\substack{P \ \text{monic} \\ \text{irreducible} \\ P\mid f}}\left(\frac{|P|}{|P|+1}\right)+O\left(|D|^{\tfrac{1}{2}}|f|^{\varepsilon}\right).
\end{equation}
\end{lemma}

\section{Proof of Theorem \ref{thm:main}}

In this section we prove Theorem \ref{thm:main}.

Let $k$ be a given even number, and set $x=\frac{2(2g)}{15k}$. We define

\begin{equation}
\label{4.1}
A(D)=\sum_{\text{deg}(n)\leq x}\frac{\chi_{D}(n)}{\sqrt{|n|}},
\end{equation}
and let 

\begin{equation}
\label{4.2}
S_{1}:=\sum_{D\in\mathcal{H}_{2g+1,q}}L(\tfrac{1}{2},\chi_{D})A(D)^{k-1},
\end{equation}
and

\begin{equation}
\label{4.3}
S_{2}:=\sum_{D\in\mathcal{H}_{2g+1,q}}A(D)^{k}.
\end{equation}

An application of Triangle inequality followed by H\"{o}lder's inequality gives us that,

\begin{align}
\label{4.4}
& \Bigg|\sum_{D\in\mathcal{H}_{2g+1,q}}L(\tfrac{1}{2},\chi_{D})A(D)^{k-1}\Bigg|\leq \sum_{D\in\mathcal{H}_{2g+1,q}}|L(\tfrac{1}{2},\chi_{D})||A(D)|^{k-1} \nonumber \\
 &\leq\left(\sum_{D\in\mathcal{H}_{2g+1,q}}L(\tfrac{1}{2},\chi_{D})^{k}\right)^{1/k}\left(\sum_{D\in\mathcal{H}_{2g+1,q}}A(D)^{k}\right)^{\tfrac{k-1}{k}}.
\end{align}
From \eqref{4.4} we have

\begin{align}
\label{4.5}
\sum_{D\in\mathcal{H}_{2g+1,q}}L(\tfrac{1}{2},\chi_{D})^{k}&\geq\frac{\left(\sum_{D\in\mathcal{H}_{2g+1,q}}L(\tfrac{1}{2},\chi_{D})A(D)^{k-1}\right)^{k}}{\left(\sum_{D\in\mathcal{H}_{2g+1,q}}A(D)^{k}\right)^{k-1}}\nonumber \\
&=\frac{S_{1}^{k}}{S_{2}^{k-1}}.
\end{align}

Hence from \eqref{4.5} we can see that to prove Theorem \ref{thm:main} we only need to give satisfactory estimates for $S_{1}$ and $S_{2}$. We start with $S_{2}$.

\subsection{\textbf{Estimating $S_{2}$}}

We have that

\begin{equation}
\label{4.6}
A(D)^{k}=\sum_{\substack{n_{1},\ldots,n_{k} \\ \text{deg}(n_{j})\leq x \\ j=1,\ldots,k}}\frac{\chi_{D}(n_{1}\ldots n_{k})}{\sqrt{|n_{1}|\ldots|n_{k}|}}.
\end{equation}
So,

\begin{align}
\label{4.7}
S_{2}&=\sum_{D\in\mathcal{H}_{2g+1,q}}A(D)^{k}\nonumber\\
&=\sum_{\substack{n_{1},\ldots,n_{k} \\ \text{deg}(n_{j})\leq x \\ j=1,\ldots,k}}\frac{1}{\sqrt{|n_{1}|\ldots|n_{k}|}}\sum_{D\in\mathcal{H}_{2g+1,q}}\left(\frac{D}{n_{1}\ldots n_{k}}\right).
\end{align}

At this stage we need an auxiliary Lemma. It is called orthogonal relations for quadratic characters and it has appeared in a different form in \cite{And-Kea,And-Kea1,Fai-Rud}.

\begin{lemma}
\label{lem4}
If $n\in A$ is not a perfect square then

\begin{equation}
\label{4.8}
\sum_{\substack{D\in\mathcal{H}_{2g+1,q} \\ n\neq\square}}\left(\frac{D}{n}\right)\ll|D|^{1/2}|n|^{1/4}.
\end{equation}
And if $n\in A$ is a perfect square then

\begin{equation}
\label{4.9}
\sum_{\substack{D\in\mathcal{H}_{2g+1,q} \\ n=\square}}\left(\frac{D}{n}\right)=\frac{|D|}{\zeta_{A}(2)}\prod_{\substack{P \ \text{monic} \\ \text{irreducible} \\ P\mid n}}\left(\frac{|P|}{|P|+1}\right)+O\left(|D|^{\tfrac{1}{2}}|n|^{\varepsilon}\right),
\end{equation}
for any $\varepsilon>0$.
\end{lemma}

\begin{remark}
Equation \eqref{4.8} can be seen as an improvement on the estimate given in \cite[Lemma 3.1]{Fai-Rud}. And the same equation \eqref{4.8} can be used to improve the error term in the first moment of quadratic Dirichlet $L$--functions over function fields as given in \cite[Theorem 2.1]{And-Kea}.
\end{remark}

\begin{proof}
If $n=\square$, then

\begin{equation}
\label{4.10}
\sum_{\substack{D\in\mathcal{H}_{2g+1,q} \\ n=\square}}\left(\frac{D}{n}\right)=\sum_{\substack{D\in\mathcal{H}_{2g+1,q} \\ (D,n)=1}}1.
\end{equation}
By invoking Lemma \ref{lem3} we establish equation \eqref{4.9}.

For \eqref{4.8} we write

\begin{align}
\label{4.11}
&\sum_{\substack{D\in\mathcal{H}_{2g+1,q}}}\left(\frac{D}{n}\right)=\sum_{2\alpha+\beta=2g+1}\sum_{\text{deg}(B)=\beta}\sum_{\text{deg}(A)=\alpha}\mu(A)\left(\frac{A^{2}B}{n}\right)\nonumber\\
&=\sum_{0\leq\alpha\leq g}\sum_{\text{deg}(A)=\alpha}\mu(A)\left(\frac{A^{2}}{n}\right)\sum_{\text{deg}(B)=2g+1-2\alpha}\left(\frac{B}{n}\right)\nonumber\\
&\leq\sum_{0\leq\alpha\leq g}\sum_{\text{deg}(A)=\alpha}\sum_{\text{deg}(B)=2g+1-2\alpha}\left(\frac{B}{n}\right).
\end{align}
If $n\neq\square$ then $\sum_{\text{deg}(B)=2g+1-2\alpha}\left(\frac{B}{n}\right)$ is a character sum to a non--principal character modulo $n$. So using Lemma \ref{lem2} we have that

\begin{equation}
\label{4.12}
\sum_{\text{deg}(B)=2g+1-2\alpha}\left(\frac{B}{n}\right)\ll|n|^{1/2}.
\end{equation}
Further we can estimate trivially the non--principal character sum by

\begin{equation}
\label{4.13}
\sum_{\text{deg}(B)=2g+1-2\alpha}\left(\frac{B}{n}\right)\ll\frac{|D|}{|A|^{2}}=q^{2g+1-2\alpha}.
\end{equation}
Thus, if $n\neq\square$, we obtain that 

\begin{align}
\label{4.14}
\sum_{D\in\mathcal{H}_{2g+1,q}}\left(\frac{D}{n}\right)&\ll\sum_{0\leq\alpha\leq g}\sum_{\text{deg}(A)=\alpha}\text{min}\left(|n|^{1/2},\frac{|D|}{|A|^{2}}\right)\nonumber \\
&\ll|D|^{\tfrac{1}{2}}|n|^{\tfrac{1}{4}},
\end{align}
upon using the first bound \eqref{4.12} for $\alpha\leq g-\frac{\text{deg}(n)}{4}$ and the second bound \eqref{4.13} for larger $\alpha$. And this concludes the proof of the lemma.
\end{proof}

Using Lemma \ref{lem4} in \eqref{4.7} we obtain that

\begin{align}
\label{4.15}
S_{2}&=\sum_{\substack{n_{1},\ldots,n_{k} \\ \text{deg}(n_{j})\leq x \\ j=1,\ldots,k \\ n_{1}\ldots n_{k}=\square}}\frac{1}{\sqrt{|n_{1}|\ldots|n_{k}|}}\Bigg(\frac{|D|}{\zeta_{A}(2)}\prod_{\substack{P \ \text{monic} \\ \text{irreducible} \\ P\mid n_{1}\ldots n_{k}}}\left(\frac{|P|}{|P|+1}\right)\Bigg)\nonumber\\
&+\sum_{\substack{n_{1},\ldots,n_{k} \\ \text{deg}(n_{j})\leq x \\ j=1,\ldots,k \\ n_{1}\ldots n_{k}=\square}}\frac{1}{\sqrt{|n_{1}|\ldots|n_{k}|}}O\left(|D|^{\tfrac{1}{2}}|n_{1}\ldots n_{k}|^{\varepsilon}\right)\nonumber\\
&+\sum_{\substack{n_{1},\ldots,n_{k} \\ \text{deg}(n_{j})\leq x \\ j=1,\ldots,k \\ n_{1}\ldots n_{k}\neq\square}}\frac{1}{\sqrt{|n_{1}|\ldots|n_{k}|}}O\left(|D|^{\tfrac{1}{2}}|n_{1}\ldots n_{k}|^{\tfrac{1}{4}}\right).
\end{align}
After some arithmetic manipulations with the $O$--terms we get that

\begin{align}
\label{4.16}
S_{2}&=\frac{|D|}{\zeta_{A}(2)}\sum_{\substack{n_{1},\ldots,n_{k} \\ \text{deg}(n_{j})\leq x \\ j=1,\ldots,k \\ n_{1}\ldots n_{k}=\square}}\frac{1}{\sqrt{|n_{1}|\ldots|n_{k}|}}\prod_{\substack{P \ \text{monic} \\ \text{irreducible} \\ P\mid n_{1}\ldots n_{k}}}\left(\frac{|P|}{|P|+1}\right)\nonumber \\
&+O\left(|D|^{\tfrac{1}{2}}q^{\left(\tfrac{3}{4}+\varepsilon\right)x}\right).
\end{align}
Since $x=\frac{2(2g)}{15k}$, the error term above is $\ll|D|^{\tfrac{3}{5}}$. So,

\begin{equation}
\label{4.17}
S_{2}=\frac{|D|}{\zeta_{A}(2)}\sum_{\substack{n_{1},\ldots,n_{k} \\ \text{deg}(n_{j})\leq x \\ j=1,\ldots,k \\ n_{1}\ldots n_{k}=\square}}\frac{1}{\sqrt{|n_{1}|\ldots|n_{k}|}}\prod_{\substack{P \ \text{monic} \\ \text{irreducible} \\ P\mid n_{1}\ldots n_{k}}}\left(\frac{|P|}{|P|+1}\right)+O\left(|D|^{\tfrac{3}{5}}\right).
\end{equation}

Writing $n_{1}\ldots n_{k}=m^{2}$ we see that

\begin{align}
\label{4.18}
&\sum_{\substack{m^{2} \ \text{monic} \\ \text{deg}(m^{2})\leq x}}\frac{d_{k}(m^{2})}{|m|}\prod_{\substack{P \ \text{monic} \\ \text{irreducible} \\ P\mid m}}\left(\frac{|P|}{|P|+1}\right)\nonumber \\
&\leq\sum_{\substack{n_{1},\ldots,n_{k} \\ \text{deg}(n_{j})\leq x \\ j=1,\ldots,k \\ n_{1}\ldots n_{k}=\square=m^{2}}}\frac{1}{\sqrt{|n_{1}|\ldots|n_{k}|}}\prod_{\substack{P \ \text{monic} \\ \text{irreducible} \\ P\mid n_{1}\ldots n_{k}}}\left(\frac{|P|}{|P|+1}\right)\nonumber\\
&\leq\sum_{\substack{m^{2} \ \text{monic} \\ \text{deg}(m^{2})\leq kx}}\frac{d_{k}(m^{2})}{|m|}\prod_{\substack{P \ \text{monic} \\ \text{irreducible} \\ P\mid m}}\left(\frac{|P|}{|P|+1}\right),
\end{align} 
where $d_{k}(m)$ represents the number of ways to write the monic polynomial $m$ as a product of $k$ factors.


We need to obtain an estimate for

\begin{equation}
\label{4.19}
\sum_{\substack{m \ \text{monic} \\ \text{deg}(m)=x}}d_{k}(m^{2})a_{m},
\end{equation}
where $a_{m}=\prod_{\substack{P \ \text{monic} \\ \text{irreducible} \\ P\mid m}}\left(\frac{|P|}{|P|+1}\right)$.

To obtain the desired estimate we consider the corresponding Dirichlet series

\begin{equation}
\label{4.20}
\zeta_{f}(s)=\sum_{m \ \text{monic}}\frac{d_{k}(m^{2})a_{m}}{|m|^{s}}=\sum_{n=0}^{\infty}\sum_{\text{deg}(m)=x}d_{k}(m^{2})a_{m}u^{x}=Z_{f}(u),
\end{equation}
with $u=q^{-s}$. Writing the above as an Euler product

\begin{equation}
\label{4.21}
\zeta_{f}(s)=\prod_{\substack{P \ \text{monic}  \\ \text{irreducible}}}\left(1+\frac{d_{k}(P^{2})a_{P}}{|P|^{s}}+\frac{d_{k}(P^{4})a_{P^{2}}}{|P|^{2s}}+\cdots\right),
\end{equation}
we can identify the poles of $\zeta_{f}(s)$. Similar calculations carried out in the classical case by Soundararajan and Rudnick \cite[page 9]{Ru-So2} and Selberg \cite[Theorem 2]{Sel}, and for function fields by Andrade and Keating \cite[Section 4.3]{And-Kea1} shows us that $\zeta_{f}(s)$ has a pole at $s=1$ of order $\frac{k(k+1)}{2}$. Therefore we can write

\begin{align}
\label{4.22}
\zeta_{f}(s)&=\prod_{\substack{P \ \text{monic}  \\ \text{irreducible}}}\left(1-\frac{1}{|P|^{s}}\right)^{-\tfrac{k(k+1)}{2}}\nonumber \\
&\times\prod_{\substack{P \ \text{monic}  \\ \text{irreducible}}}\left(1+\left(\frac{|P|}{|P|+1}\sum_{j=1}^{\infty}\frac{d_{k}(P^{2j})}{|P|^{js}}\right)\right)\left(1-\frac{1}{|P|^{s}}\right)^{\tfrac{k(k+1)}{2}},
\end{align}
where the first product has a pole at $s=1$ of order $\frac{k(k+1)}{2}$ and the second product above \eqref{4.22} is convergent for $\text{Re}(s)>1$ and holomorphic in $\{s\in B \ | \ \text{Re}(s)=1\}$ with 

\begin{equation}
\label{4.23}
B=\left\{s\in\mathbb{C} \ | \ -\frac{\pi i}{\log(q)}\leq\mathfrak{I}(s)<\frac{\pi i}{\log(q)}\right\}.
\end{equation}
Thus we can use Theorem 17.4 from \cite{Ros} to obtain the desired estimate. But we sketch below how this can be done. A standard contour integration (Cauchy's theorem)

\begin{equation}
\label{4.24}
\frac{1}{2\pi i}\oint_{C_{\varepsilon}+C}{\frac{Z_{f}(u)}{u^{x+1}}du}=\sum\text{Res}(Z_{f}(u)u^{-x-1}),
\end{equation}
where $C$ is the boundary of the disc $\{u\in\mathbb{C} \ | \ |u|\leq q^{-\delta}\}$ for some $\delta<1$ and $C_{\varepsilon}$ a small circle about $s=0$ oriented clockwise. There is only one pole in the integration region $C_{\varepsilon}+C$ and it is located at $u=q^{-1}$ as can be seen from \eqref{4.22}. To find the residue there, we expand both $Z_{f}(u)$ and $u^{-x-1}$ in Laurent series about $u=q^{-1}$, multiply the results together, and pick out the coefficient of $(u-q^{-1})^{-1}$. After this residue calculation we obtain that

\begin{equation}
\label{4.25}
\sum_{\substack{m \ \text{monic} \\ \text{deg}(m)=x}}d_{k}(m^{2})a_{m}\sim C(k)q^{x}x^{\tfrac{k(k+1)}{2}-1},
\end{equation}
for a positive constant $C(k)$ explicitly given by

\begin{equation}
\label{4.26}
C(k)=\frac{\log(q)^{\tfrac{k(k+1)}{2}}}{\left(\tfrac{k(k+1)}{2}-1\right)!}\alpha,
\end{equation}
with

\begin{equation}
\label{4.27}
\alpha=\lim_{s\rightarrow1}\left[(s-1)^{\tfrac{k(k+1)}{2}}\zeta_{f}(s)\right].
\end{equation}

In the end we obtain that

\begin{equation}
\label{4.28}
\sum_{\substack{m \ \text{monic} \\ \text{deg}(m)\leq z}}\frac{d_{k}(m^{2})}{|m|}\prod_{\substack{P \ \text{monic} \\ \text{irreducible} \\ P\mid m}}\left(\frac{|P|}{|P|+1}\right)\sim C(k)(z)^{k(k+1)/2}.
\end{equation}
Therefore we can conclude that

\begin{equation}
\label{4.29}
S_{2}\ll|D|(\log_{q}|D|)^{k(k+1)/2}.
\end{equation}

\subsection{\textbf{Estimating $S_{1}$.}}

It remains to evaluate $S_{1}$ and for that we need an ``approximate" functional equation for $L(\tfrac{1}{2},\chi_{D})$. Using Lemma \ref{lem1} with $s=\frac{1}{2}$ we have that

\begin{align}
\label{4.30}
S_{1}&=\sum_{D\in\mathcal{H}_{2g+1,q}}\left(\sum_{\text{deg}(f_{1})\leq g}\frac{\chi_{D}(f_{1})}{|f_{1}|^{1/2}}+\sum_{\text{deg}(f_{2})\leq g-1}\frac{\chi_{D}(f_{2})}{|f_{2}|^{1/2}}\right)\nonumber \\
&\times\Bigg(\sum_{\substack{n_{1},\ldots,n_{k-1} \\ \text{deg}(n_{j})\leq x \\ j=1,\ldots,k-1}}\frac{\chi_{D}(n_{1}\ldots n_{k-1})}{\sqrt{|n_{1}|\ldots|n_{k-1}|}}\Bigg)\nonumber \\
&=\sum_{\text{deg}(f_{1})\leq g}\frac{1}{\sqrt{|f_{1}|}}\sum_{\substack{n_{1},\ldots,n_{k-1} \\ \text{deg}(n_{j})\leq x \\ j=1,\ldots,k-1}}\frac{1}{\sqrt{|n_{1}|\ldots|n_{k-1}|}}\sum_{D\in\mathcal{H}_{2g+1,q}}\left(\frac{D}{f_{1}n_{1}\ldots n_{k-1}}\right)\nonumber \\
&+\sum_{\text{deg}(f_{2})\leq g-1}\frac{1}{\sqrt{|f_{2}|}}\sum_{\substack{n_{1},\ldots,n_{k-1} \\ \text{deg}(n_{j})\leq x \\ j=1,\ldots,k-1}}\frac{1}{\sqrt{|n_{1}|\ldots|n_{k-1}|}}\sum_{D\in\mathcal{H}_{2g+1,q}}\left(\frac{D}{f_{2}n_{1}\ldots n_{k-1}}\right).
\end{align}

In the last equality in equation \eqref{4.30} the sums over $f_{1}$ and $f_{2}$ are exactly the same, with the only difference being the size of the sums, i.e., $\text{deg}(f_{1})\leq g$ and $\text{deg}(f_{2})\leq g-1$. We estimate only the $f_{1}$ sum in the last equality and the result being the same for the $f_{2}$ sum just replacing $g$ by $g-1$.

If $f_{1}n_{1}\ldots n_{k-1}$ is not a square then an application of Lemma \ref{lem4} gives us that

\begin{align}
\label{4.31}
&\sum_{\text{deg}(f_{1})\leq g}\frac{1}{\sqrt{|f_{1}|}}\sum_{\substack{n_{1},\ldots,n_{k-1} \\ \text{deg}(n_{j})\leq x \\ j=1,\ldots,k-1}}\frac{1}{\sqrt{|n_{1}|\ldots|n_{k-1}|}}\sum_{D\in\mathcal{H}_{2g+1,q}}\left(\frac{D}{f_{1}n_{1}\ldots n_{k-1}}\right)\nonumber \\
&\ll\sum_{\text{deg}(f_{1})\leq g}\frac{1}{\sqrt{|f_{1}|}}\sum_{\substack{n_{1},\ldots,n_{k-1} \\ \text{deg}(n_{j})\leq x \\ j=1,\ldots,k-1}}\frac{1}{\sqrt{|n_{1}|\ldots|n_{k-1}|}}|D|^{\tfrac{1}{2}}|f_{1}n_{1}\ldots n_{k-1}|^{\tfrac{1}{4}}\nonumber\\
&=|D|^{\tfrac{1}{2}}\sum_{\text{deg}(f_{1})\leq g}|f_{1}|^{-\tfrac{1}{4}}\sum_{\text{deg}(n_{1})\leq x}|n_{1}|^{-\tfrac{1}{4}}\cdots\sum_{\text{deg}(n_{k-1})\leq x}|n_{k-1}|^{-\tfrac{1}{4}}\nonumber \\
&\ll|D|^{\tfrac{1}{2}}q^{\tfrac{3}{4}g}(q^{x})^{(k-1)\tfrac{3}{4}}.
\end{align}
With our choice of $x$, we have that for $f_{1}n_{1}\ldots n_{k-1}$ not a square

\begin{align}
\label{4.32}
&\sum_{\text{deg}(f_{1})\leq g}\frac{1}{\sqrt{|f_{1}|}}\sum_{\substack{n_{1},\ldots,n_{k-1} \\ \text{deg}(n_{j})\leq x \\ j=1,\ldots,k-1}}\frac{1}{\sqrt{|n_{1}|\ldots|n_{k-1}|}}\sum_{D\in\mathcal{H}_{2g+1,q}}\left(\frac{D}{f_{1}n_{1}\ldots n_{k-1}}\right)\nonumber \\
&\ll|D|^{\tfrac{39}{40}}.
\end{align}

For $f_{2}n_{1}\ldots n_{k-1}$ not equal to a perfect square, the same reasoning gives

\begin{align}
\label{4.33}
&\sum_{\text{deg}(f_{2})\leq g-1}\frac{1}{\sqrt{|f_{2}|}}\sum_{\substack{n_{1},\ldots,n_{k-1} \\ \text{deg}(n_{j})\leq x \\ j=1,\ldots,k-1}}\frac{1}{\sqrt{|n_{1}|\ldots|n_{k-1}|}}\sum_{D\in\mathcal{H}_{2g+1,q}}\left(\frac{D}{f_{2}n_{1}\ldots n_{k-1}}\right)\nonumber \\
&\ll|D|^{\tfrac{39}{40}}.
\end{align}

It remains to estimate the main--term in $S_{1}$. If $f_{1}n_{1}\ldots n_{k-1}$ is a perfect square then

\begin{multline}
\label{4.34}
\sum_{D\in\mathcal{H}_{2g+1,q}}\sum_{\substack{f_{1},n_{1},\ldots,n_{k-1} \\ \text{deg}(f_{1})\leq g \\ \text{deg}(n_{j})\leq x \\ j=1,\ldots,k-1}}\frac{1}{\sqrt{|f_{1}||n_{1}|\ldots|n_{k-1}|}}\chi_{D}(f_{1}n_{1}\ldots n_{k-1}) \\
=\sum_{\substack{f_{1},n_{1},\ldots,n_{k-1} \\ \text{deg}(f_{1})\leq g \\ \text{deg}(n_{j})\leq x \\ j=1,\ldots,k-1}}\frac{1}{\sqrt{|f_{1}||n_{1}|\ldots|n_{k-1}|}}\sum_{D\in\mathcal{H}_{2g+1,q}}\chi_{D}(f_{1}n_{1}\ldots n_{k-1}).
\end{multline}

By Lemma \ref{lem3} we have that \eqref{4.34} becomes

\begin{align}
\label{4.35}
&\sum_{D\in\mathcal{H}_{2g+1,q}}\sum_{\substack{f_{1},n_{1},\ldots,n_{k-1} \\ \text{deg}(f_{1})\leq g \\ \text{deg}(n_{j})\leq x \\ j=1,\ldots,k-1}}\frac{1}{\sqrt{|f_{1}||n_{1}|\ldots|n_{k-1}|}}\chi_{D}(f_{1}n_{1}\ldots n_{k-1})\nonumber \\
&=\sum_{\substack{f_{1},n_{1},\ldots,n_{k-1} \\ \text{deg}(f_{1})\leq g \\ \text{deg}(n_{j})\leq x \\ j=1,\ldots,k-1}}\frac{1}{\sqrt{|f_{1}||n_{1}|\ldots|n_{k-1}|}}\frac{|D|}{\zeta_{A}(2)}\prod_{\substack{P \ \text{monic} \\ \text{irreducible} \\ P\mid f_{1}n_{1}\ldots n_{k-1}}}\left(\frac{|P|}{|P|+1}\right)\nonumber \\
&+O\Bigg(\sum_{\substack{f_{1},n_{1},\ldots,n_{k-1} \\ \text{deg}(f_{1})\leq g \\ \text{deg}(n_{j})\leq x \\ j=1,\ldots,k-1}}\frac{1}{\sqrt{|f_{1}||n_{1}|\ldots|n_{k-1}|}}|D|^{\tfrac{1}{2}}|f_{1}n_{1}\ldots n_{k-1}|^{\varepsilon}\Bigg)
\end{align}

If we call $a_{f}=\prod_{P\mid f}\left(\frac{|P|}{|P|+1}\right)$, then we have

\begin{align}
\label{4.36}
&\sum_{D\in\mathcal{H}_{2g+1,q}}\sum_{\substack{f_{1},n_{1},\ldots,n_{k-1} \\ \text{deg}(f_{1})\leq g \\ \text{deg}(n_{j})\leq x \\ j=1,\ldots,k-1}}\frac{1}{\sqrt{|f_{1}||n_{1}|\ldots|n_{k-1}|}}\chi_{D}(f_{1}n_{1}\ldots n_{k-1})\nonumber \\
&=\frac{|D|}{\zeta_{A}(2)}\sum_{\text{deg}(m)\leq\frac{g+(k-1)x}{2}}\frac{a_{m^{2}}d_{k}(m^{2})}{|m|}+O\left(|D|^{\tfrac{1}{2}}q^{g(\varepsilon-\tfrac{1}{2})+g}(q^{x(\varepsilon-\tfrac{1}{2})+x})^{k-1}\right).
\end{align}
With our choice of $x$ we have that the $O$--term above is $\ll|D|^{39/40}$.

The last step is to estimate the main term contribution

\begin{equation}
\label{4.37}
\frac{|D|}{\zeta_{A}(2)}\sum_{\text{deg}(m)\leq\frac{g+(k-1)x}{2}}\frac{a_{m^{2}}d_{k}(m^{2})}{|m|}.
\end{equation}
By employing the same reasoning of Rudnick and Soundararajan \cite[page 10]{Ru-So2} we write $n_{1}\cdots n_{k-1}=rh^{2}$ where $r$ and $h$ are monic polynomials and $r$ is square--free. Then $f_{1}$ is of the form $rl^{2}$. With this notation the main term contribution is

\begin{equation}
\label{4.38}
\frac{|D|}{\zeta_{A}(2)}\sum_{\substack{n_{1},\ldots,n_{k-1} \\ n_{1}\cdots n_{k-1}=rh^{2} \\ \text{deg}(n_{j})\leq x \\ j=1,\ldots,k-1}}\frac{1}{|rh|}\sum_{\substack{l \ \text{monic} \\ \text{deg}(l)\leq\frac{g-\text{deg}(r)}{2}}}\frac{1}{|l|}a_{rhl}.
\end{equation}

Note that $\text{deg}(r)\leq (k-1)x$ and an easy calculation as those used in \cite[page 10]{Ru-So2} and \cite[Lemma 5.7 and pages 2812--2813]{And-Kea} gives that the sum over $l$ above is

\begin{equation}
\label{4.39}
\sum_{\substack{l \ \text{monic} \\ \text{deg}(l)\leq\frac{g-\text{deg}(r)}{2}}}\frac{1}{|l|}a_{rhl}\sim C(r,h)a_{rh}g,
\end{equation}
for some positive constant $C(r,h)$.

Therefore follows that the main term contribution to \eqref{4.36} is

\begin{align}
\label{4.40}
&\gg|D|(\log_{q}|D|)\sum_{\substack{n_{1},\ldots,n_{k-1} \\ n_{1}\cdots n_{k-1}=rh^{2} \\ \text{deg}(n_{j})\leq x \\ j=1,\ldots,k-1}}\frac{1}{|rh|}a_{rh}\nonumber \\
&\gg|D|(\log_{q}|D|)\sum_{\substack{r,h \\ \text{deg}(rh^{2})\leq x}}\frac{d_{k-1}(rh^{2})}{|rh|}a_{rh}\nonumber \\
&\gg|D|(\log_{q}|D|)^{k(k+1)/2},
\end{align} 
where the last bound follows by activating the same estimate as proved in past section, replacing $k$ by $k-1$. The same argument applies to the second sum in \eqref{4.29} replacing $g$ by $g-1$. Therefore we can conclude that 

\begin{equation}
\label{4.41}
S_{1}\gg|D|(\log_{q}|D|)^{k(k+1)/2}.
\end{equation}

Combining \eqref{4.29} and \eqref{4.41} finishes the prove of Theorem \ref{thm:main}.

\vspace{0.5cm}

\noindent \textit{Acknowledgement.} I am very happy to thank the American Institute of Mathematics (AIM) where this work was finished during the workshop ``Arithmetic Statistics over Finite Fields and Function Fields 2014".

\noindent I would like to express my gratitude to Professor Ze\'{e}v Rudnick for his constant encouragement and for his helpful comments in an earlier draft of this manuscript. I also would like to thank to the comments of an anonymous referee.

\noindent The author also would like to thank the support of a William Hodge Fellowship (EPSRC) and an IH\'{E}S Postdoctoral Research Fellowship. This research was also supported by EPSRC grant EP/K021132X/1.

\end{document}